\documentclass[11pt,a4paper]{amsart}

\usepackage[cp1250]{inputenc}
\usepackage[T1]{fontenc}
\usepackage{amssymb,amsmath}
\usepackage[mathscr]{eucal}
\usepackage{amsxtra}
\usepackage{graphics, color,bm}
\usepackage{url}
\usepackage{verbatim}
\usepackage{yhmath}
\usepackage{lmodern}

\newtheorem{theorem}{Theorem}
\newtheorem{proposition}[theorem]{Proposition}
\newtheorem{lemma}[theorem]{Lemma}
\newtheorem{corollary}[theorem]{Corollary}

\newtheorem{example}[theorem]{Example}
	\theoremstyle{definition}

\newcommand{\Refk}{{\rm Ref_{k}}}
\newcommand{\rd}{{\rm rd}}
\newcommand{\tr}{{\rm tr}}
\newcommand{\im}{{\rm im}}

\newcommand{\bC}{\mathbb{C}}
\newcommand{\bM}{\mathbb{M}}
\newcommand{\bN}{\mathbb{N}}

\newcommand{\pF}{\mathcal{F}}
\newcommand{\pS}{\mathcal{S}}

\newcommand{\pX}{\mathcal{X}}

\begin{document}


\title[$k$-reflexivity defect of the image of a generalized derivation]{$k$-reflexivity defect of the image of a generalized derivation}
\author[T. Rudolf]{Tina Rudolf}
\address{University of Ljubljana, IMFM, Jadranska ul. 19, 1000 Ljubljana, Slovenia}
\email{tina.rudolf@fmf.uni-lj.si}


\begin{abstract}
Let $\mathcal{X}$ be a finite-dimensional complex vector space and let $k$ be a positive integer. An explicit formula for the $k$-reflexivity defect of the image of a generalized derivation on $L(\pX)$, the space of all linear transformations on $\mathcal{X}$, is given. Using latter, we also study the $k$-reflexivity defect of the image of an elementary operator of the form $\Delta(T)=ATB-T$ ($T \in L(\mathcal{X})$).
\end{abstract}

\maketitle

\section{Introduction}
\setcounter{theorem}{0}

Let $\mathcal{X}$ be a finite-dimensional complex vector space and let $L(\mathcal{X})$ be the space of all linear transformations on $\mathcal{X}$. Let $k$ be a positive integer and denote by $\pF_k$ the set of all elements in $\bM_n$ of rank $k$ or less. The $k$-reflexive cover of a non-empty subset $\mathcal{S} \subseteq L(\mathcal{X})$ is defined by
\begin{equation*}
{\rm Ref_{k}}  \mathcal{S}=\{T \in L(\mathcal{X}): \forall \varepsilon >0, \, \forall x_1, \ldots,x_k \in \mathcal{X},\, \exists S \in \mathcal{S}: \, \|Tx_i-Sx_i\|<\varepsilon, \, i=1,\ldots,k\}.
\end{equation*}
It is easy to see that ${\rm Ref_{k}}  \mathcal{S}$ is a linear subspace of $L(\mathcal{X})$. A linear subspace $\mathcal{S}$ is said to be $k$-reflexive if ${\rm Ref_{k}}  \mathcal{S}=\mathcal{S}$. The $k$-reflexivity defect of a non-empty subset $\pS$ is defined by ${\rm rd} _k(\mathcal{S})=\dim ({\rm Ref_{k}}  \mathcal{S} / \mathcal{S})$. Since $\pX$ is finite dimensional ${\rm rd} _k(\mathcal{S})=\dim ({\rm Ref_{k}}  \mathcal{S})-\dim (\mathcal{S})$ holds. The annihilator of a non-empty subset $\pS \subseteq \bM_n$ is defined by $\pS_\perp=\{C \in \bM_n:\, \tr (CS)=0 \ \textup{for all} \ S \in \pS\}$, where $\tr(\cdot)$ denotes the trace functional. It was shown in \cite{KL, KL2} that
\begin{equation}	\label{k-ref}
\Refk \pS=\left(\pS_\perp \cap \pF_k \right)_\perp
\end{equation}
holds. The latter obviously implies that a $k$-reflexive space is also $j$-reflexive for all $j \geq k$. 

Let $A,\,B \in L(\mathcal{X})$ be invertible linear transformations and let $\mathcal{S}$ be a linear subspace of $L(\mathcal{X})$. Let us denote $A\mathcal{S} B=\{ASB:\,S \in \mathcal{S}\}$ and $\mathcal{S}^\intercal=\{S^\intercal:\, S \in \mathcal{S}\}$. It is well known that transformations of the type
\begin{equation}	\label{transf}
\mathcal{S} \mapsto A\mathcal{S} B = \{ASB:\,S \in \mathcal{S}\}	\quad \textrm{and} \quad \mathcal{S} \mapsto \mathcal{S}^\intercal =\{S^\intercal:\, S \in \mathcal{S}\}
\end{equation}
preserve the $k$-reflexivity defect. Hence, since $\mathcal{X}$ is a finite-dimensional complex vector space, one can assume that $\mathcal{X}=\mathbb{C}^n$ for some $n \in \mathbb{N}$ and $L(\mathcal{X})$ may be identified with $\mathbb{M}_n$, the algebra of all $n$-by-$n$ complex matrices. Throughout this paper we will be dealing with subspaces of $\mathbb{M}_n$ which have the decomposition of the form
$$\mathcal{S}=\left(\begin{array}{ccc} \mathcal{S}_{11}&\ldots&\mathcal{S}_{1N}\\\vdots&&\vdots\\\mathcal{S}_{M1}& \ldots&\mathcal{S}_{MN}\end{array}\right),$$
where, for each pair of indices $(i,\,j)$, $\mathcal{S}_{ij}$ is a subspace of $\mathbb{M}_{m_i,n_j}$, the space of all $m_i$-by-$n_j$ complex matrices, and $\sum_{i=1}^Mm_i=\sum_{j=1}^Nn_j=n$. It is not hard to see that for spaces of this type one has
\begin{equation} \label{ref}
{\rm Ref_{k}} (\mathcal{S})=\left(\begin{array}{ccc} {\rm Ref_{k}}  (\mathcal{S}_{11})&\ldots&{\rm Ref_{k}}  (\mathcal{S}_{1N})\\\vdots&&\vdots\\{\rm Ref_{k}}  (\mathcal{S}_{M1})&\ldots&{\rm Ref_{k}}  (\mathcal{S}_{MN})\end{array}\right) \qquad \textrm{and} \qquad {\rm rd}  _k\left(\mathcal{S}\right)=\sum_{i=1}^M \sum_{j=1}^N{\rm rd}  _k \left(\mathcal{S}_{ij}\right).
\end{equation}
In particular, $\mathcal{S}$ is $k$-reflexive if and only if $\mathcal{S}_{ij}$ is $k$-reflexive for every pair of indices $i\in \{1,\ldots,M\}$, $j \in \{1,\ldots,N\}$.

\section{Elementary operators}
\label{EO}
\setcounter{theorem}{0}
Let $(A_1,\ldots,A_k)$ and $(B_1,\ldots,B_k)$ be arbitrary pairs of $n$-by-$n$ complex matrices. The elementary operator on $\mathbb{M}_n$ with coefficients $(A_1,\ldots,A_k)$ and $(B_1,\ldots,B_k)$ is defined by
\begin{equation*}	\label{delta}
\Delta (T)=A_1TB_1+\ldots+A_kTB_k, \qquad T \in \mathbb{M}_n.
\end{equation*}
If all $A_i$ are pairwise linearly independent and if the same holds for all $B_i$ ($1 \leq i \leq k$), then $\Delta$ is called elementary operator of length $k$. The simplest example of such operator is of course two-sided multiplication. Namely, let $A,\,B \in \bM_n$ and let $\Delta$ be an elementary operator defined by $\Delta \left(T\right)=ATB$ for $T \in \bM_n$. It is easy to see that the kernel and the image of $\Delta$ are reflexive spaces. In fact, if $\Delta$ is an elementary operator of length $k$ on $\bM_n$, then by \cite[Proposition 1.1]{B1} the space $\ker \Delta$ is $j$-reflexive for every $j \geq k$. It is reasonable to ask whether the same holds for $\im \Delta$ and we show that this is not generally the case.

\begin{lemma} \label{anih}
Let $\Delta$ be an elementary operator on $\bM_n$ with coefficients $\left(A_1,\, \ldots,\, A_k\right)$ and \linebreak $\left(B_1,\, \ldots,\,B_k\right)$, defined by $\Delta \left(T\right)=A_1TB_1+A_2TB_2+\ldots+A_kTB_k$. Then there exists an elementary operator $\tilde{\Delta}$ such that $\left(\im \Delta\right)_\perp=\ker \tilde{\Delta}$.
\end{lemma}

\begin{proof}
Define $\tilde{\Delta}\left(T\right)=B_1TA_1+B_2TA_2+\ldots+B_kTA_k$ for $T \in \bM_n$. If $T$ is an arbitrary matrix, then $\tr (\Delta(T)C)=\tr(T(B_1CA_1+\ldots+B_kCA_k))$ and therefore $C \in \left(\im \Delta\right)_\perp$ if and only if $\tilde{\Delta}(C) \in (\bM_n)_\perp=\{0\}$, that is, $C \in \ker \tilde{\Delta}$.
\end{proof}

Next, we introduce some notation. For $k \in \mathbb{N}$ and $\alpha \in \mathbb{C}$, let $J_k(\alpha)$ denote the Jordan block of size $k$, i.e.,
$$J_k(\alpha)=\left(\begin{smallmatrix} \alpha&1&&&\\&&\ddots&\ddots&\\&&&\alpha&1\\&&&&\alpha\end{smallmatrix}\right)
 \in \mathbb{M}_k.$$
In the following example we show that for any $n \geq 3$ there exists an inner derivation $\delta$ on $\bM_n$ such that $\im \delta$ is not $(n-1)$-reflexive. Consequently, the image of such elementary operator of length $2$ is not $2$-reflexive.

\begin{example}
\rm Define $\delta\left(T\right)=J_n(0)T-TJ_n(0)$ for $T \in \bM_n$. By \eqref{k-ref}, every subspace of $\bM_n$ is $n$-reflexive, hence $\rd_n (\im \delta)=0$. It follows by Lemma \ref{anih} that $(\im \delta)_\perp$ is simply $\{J_n(0)\}'$, the commutant of the Jordan block $J_n(0)$. One can easily verify that $\{J_n(0)\}'$ is the algebra of all $n \times n$ upper triangular Toeplitz matrices which we will denote by $\mathfrak{T}_n$. Namely,
$$\left(\im \delta\right)_\perp=\left\{ \left( \begin{array}{cccccc} a_1&a_2&\ldots&\ldots&a_n\\ 0&a_1&a_2& &\vdots \\ \vdots&\ddots&\ddots&\ddots&\vdots \\ \vdots&&\ddots&\ddots&a_2 \\ 0&\ldots&\ldots&0&a_1\end{array}\right) :\, a_1,\,a_2, \ldots a_n \in \bC \right\}.$$
By \eqref{k-ref}, $\im \delta$ is not $(n-1)$-reflexive space, since $\left(\im \delta\right)_\perp \cap \pF_{n-1} \subsetneq \left(\im \delta\right)_\perp$. Note that \eqref{k-ref} also implies that $\rd_k \left(\im \delta\right)=n-k$ for $1 \leq k \leq n-1$. Indeed, $\dim (\im \delta)=n^2-\dim ((\im \delta)_\perp)=n^2-n$ and by \eqref{k-ref} we have $\dim (\Refk (\im \delta))=n^2-\dim((\im \delta)_\perp \cap \pF_k)=n^2-k$ for $1 \leq k \leq n-1$.
\end{example}

\section{Generalized derivations} \label{GD}
\setcounter{theorem}{0}
Let $\Delta$ be an elementary operator of length $2$ on $\bM_n$, i.e., a linear transformation of the form $\Delta(T)=A_1TB_1+A_2TB_2$ ($T \in \mathbb{M}_n$), where $A_1,\,A_2$ and $B_1,\,B_2$ are two pairs of linearly independent matrices. By \cite[Proposition 1.1]{B1}, one has $\rd_k(\ker \Delta)=0$ for all $k \geq 2$. In \cite{R} reflexivity of such elementary operator was studied and an explicit formula for the reflexivity defect of its kernel was given. This motivates the main subject of this paper, that is the $k$-reflexivity defect of the image of some special examples of elementary operators of length $2$.

Let $A,\,B \in \mathbb{M}_n$ be arbitrary matrices. Define the generalized derivation on $\mathbb{M}_n$ with coefficients $A$ and $B$ by $\Delta \,\left(T\right)=AT-TB$, $T \in \mathbb{M}_n$. Obviously, $\Delta$ is an example of an elementary operator of length $2$. Let $J_{p_1}(\lambda_1)\oplus \ldots \oplus J_{p_N}(\lambda_N)$ be the Jordan canonical form of $A$, where $\sum_{i=1}^Np_i=n$ and $\lambda_1, \ldots,\lambda_N$ are not necessarily distinct eigenvalues of $A$. Similarly, let $J_{r_1}(\mu_1)\oplus \ldots \oplus J_{r_M}(\mu_M)$ be the Jordan canonical form of $B$, where $\sum_{i=1}^Mr_i=n$ and $\mu_1, \ldots,\mu_M$ are not necessarily distinct eigenvalues of $B$. Let $R(i,\,j,\,k)$ be a non-negative integer defined by
$$R(i,\,j,\,k):=\left\{\begin{array}{ccl}\min \{p_i,\,r_j\}-k & : & \textup{$\lambda_i=\mu_j$ and $k<\min\{p_i,\,r_j\}$,}\\ 0 & : & \textup{$\lambda_i \neq \mu_j$ or $k \geq \min\{p_i,\,r_j\}$}.\end{array}\right.$$

\begin{proposition} \label{pp2}
With the above notation, the $k$-reflexivity defect of $\im \Delta$ can be expressed as
$$\rd_k(\im \Delta)=\sum_{i=1}^N\sum_{j=1}^M R(i,\,j,\,k).$$
In particular, $\im \Delta$ is a $k$-reflexive space if and only if all roots of the greatest common divisor of $m_A$ and $m_B$ of $A$ and $B$, respectively, are of multiplicity at most $k$.
\end{proposition}

\begin{proof} 
Let $\mathbf{0}_{p,r}$ denote the $p \times r$ zero matrix ($p,\,r \in \bN$) and let $A$ and $B$ be as before the Proposition \ref{pp2}. For $1 \leq i \leq N$ and $1 \leq j \leq M$ define the following elementary operators on $\bM_{p_i,r_j}$ and $\bM_{r_j,p_i}$, respectively,
\begin{equation*}
\begin{split}
\Delta_{p_i,r_j}(T) &= J_{p_i}(\lambda_i)T-TJ_{r_j}(\mu_j) \qquad (T \in \bM_{p_i,r_j}), \\
\Delta_{r_j,p_i}(T) &= J_{r_j}(\mu_j)T-TJ_{p_i}(\lambda_i) \qquad (T \in \bM_{r_j,p_i}).
\end{split}
\end{equation*}
Lemma \ref{anih} yields $(\im \Delta_{p_i,r_j})_\perp=\ker \Delta_{r_j,p_i}$. If $\lambda_i \neq \mu_j$, then $\Delta_{p_i,r_j}$ is bijective and $\im \Delta_{p_i,r_j}$ is a $k$-reflexive space for every $k \in \bN$. Now assume that $\lambda_i=\mu_j$. It is not hard to see that
\begin{equation*}
\begin{split}
\ker \Delta_{r_j,p_i} &=\left\{\left(\begin{array}{cc}\mathbf{0}_{r_j,p_i-r_j} &  T\end{array}\right): T \in \mathfrak{T}_{r_j} \right\} \qquad \textrm{if $r_j \leq p_i$,}\\
\ker \Delta_{r_j,p_i} &=\left\{ \left(\begin{array}{c} T \\ \mathbf{0}_{r_j-p_i,p_i}\end{array}\right): T \in \mathfrak{T}_{p_i}\right\} \qquad \textrm{if $r_j > p_i$.}
\end{split}
\end{equation*}
Let us denote $d=\min\{p_i,r_j\}$ and $D=\max\{p_i,r_j\}$. Since transformations of the type \eqref{transf} preserve $k$-reflexivity defect we can without any loss of generality assume that $\ker \Delta_{r_j,p_i}=\left\{\big(\begin{array}{cc} \mathbf{0}_{d,D-d} & T \end{array}\big): T \in \mathfrak{T}_d \right\}$ and therefore $\dim (\im \Delta)=d(D-1)$. The structure of the space $\ker \Delta_{r_j,p_i}$ yields that $\ker \Delta_{r_j,p_i} \cap \pF_k$ is a linear space with the following property. If $k \geq d$, then $\ker \Delta_{r_j,p_i} \cap \pF_k=\ker \Delta_{r_j,p_i}$. Otherwise, if $1 \leq k <d$, then
$$\ker \Delta_{r_j,p_i} \cap \pF_k=\left\{\left(\begin{array}{cc} \mathbf{0}_{k,d-k}& T \\ \mathbf{0}_{D-k,d-k}& \mathbf{0}_{D-k,k} \end{array}\right): T \in \mathfrak{T}_k\right\}.$$
Therefore, $\im \Delta_{p_i,r_j}$ is a $k$-reflexive space iff $k \geq d$ or $\lambda_i \neq \mu_j$. Otherwise, if $k<d$ and $\lambda_i=\mu_j$, one gets $\dim(\Refk (\im \Delta_{p_i,r_j}))=dD-k$. The result in general setting now follows by \eqref{ref}.
\end{proof}

Let $A,\,B \in \mathbb{M}_n$ be as before the Proposition \ref{pp2}. Let $\Delta: \bM_n \rightarrow \bM_n$ be an elementary operator defined by $\Delta (T)=ATB-T$. Let $R(i,\,j,\,k)$ be a non-negative integer defined by
$$R(i,\,j,\,k):=\left\{\begin{array}{ccl}\min \{p_i,\,r_j\}-k & : & \textup{$\lambda_i,\,\mu_j \neq 0$, $\lambda_i=\frac{1}{\mu_j}$ and $k<\min\{p_i,\,r_j\}$,}\\ 0 & : & \textup{otherwise}.\end{array}\right.$$

\begin{corollary}
With the above notation, the $k$-reflexivity defect of $\im \Delta$ can be expressed as
$$\rd_k(\im \Delta)=\sum_{i=1}^N\sum_{j=1}^M R(i,\,j,\,k).$$
\end{corollary}

\begin{proof}
Define $\Delta_{p_i,r_j}(T)=J_{p_i}(\lambda_i)TJ_{r_j}(\mu_j)-T$ for $T \in \bM_{p_i,r_j}$. By \eqref{ref} we get $\rd_k(\im \Delta)=\sum_{i=1}^N\sum_{j=1}^M \rd_k(\im \Delta_{p_i,r_j})$, hence it suffices to determine $\rd_k(\im \Delta_{p_i,r_j})$. If $\lambda_i=\mu_j=0$, then it is not hard to see that for $T=(t_{uv}) \in \bM_{p_i,r_j}$ we have
$$\Delta_{p_i,r_j}(T)=-T+\left(\begin{array}{cccc}0&t_{21}&\ldots&t_{2,r_j-1}\\\vdots&\vdots&&\vdots\\0&t_{p_i,1}&\ldots&t_{p_i,r_j-1}\\0&0&\ldots&0 \end{array}\right),$$
therefore $\im \Delta_{p_i,r_j}=\bM_{p_i,r_j}$ and $\rd_k(\im \Delta_{p_i,r_j})=0$ for every positive integer $k$. If $\lambda_i=0$ and $\mu_j \neq 0$, then $\im \Delta_{p_i,r_j}=\{XJ_{r_j}(\mu_j):\, X \in \im \tilde{\Delta}_{p_i,r_j}\}$ where $\tilde{\Delta}_{p_i,r_j}: \bM_{p_i,r_j} \rightarrow \bM_{p_i,r_j}$ is a generalized derivation of the form $\tilde{\Delta}_{p_i,r_j}(T)=J_{p_i}(0)T-TJ_{r_j}(\mu_j)^{-1}$. Thus we have $\rd_k(\im \Delta_{p_i,r_j})=\rd_k(\im \tilde{\Delta}_{p_i,r_j})$. By \cite[Example 6.2.13]{HJ1} one can easily see that inverting matrices preserves the sizes of Jordan blocks, hence Proposition \ref{pp2} yields $\rd_k(\im \Delta_{p_i,r_j})=0$. Similarly, if $\lambda_i \neq 0$ and $\mu_j=0$ or if $\lambda_i \neq 0$, $\mu_j \neq 0$ and $\lambda_i \neq \frac{1}{\mu_j}$, then again Proposition \ref{pp2} yields $\rd_k(\im \Delta_{p_i,r_j})=0$. Now assume that $\lambda_i,\,\mu_j \neq 0$ and that $\lambda_i = \frac{1}{\mu_j}$. As before, $\rd_k(\im \Delta_{p_i,r_j})=\rd_k(\im \tilde{\Delta}_{p_i,r_j})$ where $\tilde{\Delta}_{p_i,r_j}: \bM_{p_i,r_j} \rightarrow \bM_{p_i,r_j}$ is a generalized derivation of the form $\tilde{\Delta}_{p_i,r_j}(T)=J_{p_i}(\lambda_i)T-TJ_{r_j}(\mu_j)^{-1}$. Now the Proposition \ref{pp2} yields that $\rd_k(\im \tilde{\Delta}_{p_i,r_j})=0$ if $k \geq \min\{p_i,r_j\}$ and $\rd_k(\im \tilde{\Delta}_{p_i,r_j})=\min \{p_i,\,r_j\}-k$ if $k < \min\{p_i,r_j\}$. By \eqref{ref} one gets $\rd_k(\im \Delta)=\sum_{i=1}^N\sum_{j=1}^M R(i,\,j,\,k)$.
\end{proof}

\section*{Acknowledgements}
The author is thankful to the Slovenian Research Agency for their financial support.


\end{document}